\documentclass[a4paper,11pt]{amsart}
\usepackage{dsfont,hyperref,fancyhdr,mathrsfs,amsmath,amscd,amsthm,amsfonts,latexsym,amssymb,stmaryrd}
\usepackage[all]{xy}

\DeclareMathOperator{\rank}{rank}

\def \a{\alpha}

\def \b{\beta}

\def \g{\gamma}

\def \phi{\varphi}
\def \Phi{\varPhi}
\def \p{\pi}

\def \C{\mathbb{C}\,}

\def\widecheckg{g^{\hspace*{-2.5pt}\vbox to 5pt{\hbox to
0pt{\LARGE$\check{}$}}}\hspace*{2pt}}

\def\widecheckl{\lambda^{\hspace*{-3.5pt}\vbox to 8pt{\hbox to
0pt{\LARGE$\check{}$}}}\hspace*{2pt}}

\begin{document}

\title{Almost $G_2$-manifolds with\\ 
almost twistorial structures}
\author{Radu Pantilie}  
%\thanks{} 
\address{R.~Pantilie, Institutul de Matematic\u a ``Simion~Stoilow'' al Academiei Rom\^ane,
C.P. 1-764, 014700, Bucure\c sti, Rom\^ania} 
\email{\href{mailto:Radu.Pantilie@imar.ro}{Radu.Pantilie@imar.ro}} 

\newtheorem{thm}{Theorem}[section]
\newtheorem{lem}[thm]{Lemma}
\newtheorem{cor}[thm]{Corollary}
\newtheorem{prop}[thm]{Proposition}

\theoremstyle{definition}

\newtheorem{defn}[thm]{Definition}
\newtheorem{rem}[thm]{Remark}
\newtheorem{exm}[thm]{Example}

\numberwithin{equation}{section}

\thispagestyle{empty}

\begin{abstract}
We give the necessary and sufficient conditions for the Penrose-Ward transformation to work on almost $G_2$-manifolds, 
endowed with natural almost twistorial structures.
\end{abstract} 

\maketitle 

\section*{Introduction}  
  
\indent 
This paper grew out, in part, of the fact that the method we used to characterise, in \cite{Pan-eholon}\,, the integrability of almost twistorial structures  
is flawed (see Remark \ref{rem:corrigenda}\,, below). On the other hand, we continue (see, also, \cite{Pan-dim6}\,) our study of the differential geometry related 
to $G_2$\,, the simply-connected simple complex Lie group of dimension $14$ (and to its compact real form).\\ 
\indent 
The main source of Euclidean twistorial structures \cite{DesLouPan} is provided by the closed orbits, on the Grassmannians of (co)isotropic spaces, 
of the complexification of a compact Lie group, endowed with a faithfull orthogonal representation. For $G_2$\,, there are two `fundamental' such 
(generalized) Grassmannians: the hyperquadric $Q$ in the projectivisation of the space $U$ of complex imaginary octonions, and the space $Y$ of anti-self-dual spaces in $U$.  
See Section \ref{section:iso_octo} from which it, also, follows that $G_2$ has two orbits on each of ${\rm Gr}_k^0(U)$\,, $k=2,3$\,, with the closed orbits 
given by $Y$ and $Q$\,, respectively, the latter, thus, leading to the canonical Euclidean twistorial structures on $U$  (cf.~\cite{Pan-eholon}\,). 
As the dual of this is not maximal (in the sense of \cite{DesLouPan}\,) we are led to, also, consider an Euclidean twistorial structure on $\C\!\times U$. 
See Section \ref{section:aG2_atwist}\,, where the integrability of the obtained canonical almost twistorial structures is studied, by using \cite{Bott-BBW}\,.

\section{Isotropic spaces closed under the octonionic cross product} \label{section:iso_octo} 

\indent 
We work in the complex analytic category. Let $\mathfrak{g}_2$ be the (complex) simple Lie algebra of dimension $14$\,, and let $G_2$ 
be the simply-connected simple Lie group whose Lie algebra is $\mathfrak{g}_2$\,.\\  
\indent 
Let $U$ be the space of imaginary (complex) octonions. Let $Q\subseteq PU$ be the quadric of isotropic directions,  
and let $Q'$ be the space of projective planes contained by $Q$. Recall (see \cite{Pan-eholon}\,) that $Q$ can be embedded into $Q'$ 
as an orbit of $G_2\,(\subseteq{\rm Spin}(7)$\,). 

\begin{prop} \label{prop:3_isotropic} 
A three-dimensional isotropic subspace $p\subseteq U$ 
is closed under the octonionic cross product if and only if $p\in Q(\subseteq Q')$\,.  
\end{prop} 
\begin{proof}
We claim that the action of $G_2$ on $Q'$ has two orbits, as follows:\\ 
\indent 
\quad(1) $Q$;\\ 
\indent 
\quad(2) $(\pm{\rm i})$-eigenspaces of the octonionic cross multiplication with a unit vector (this orbit 
can be identified with $G_2/{\rm SL}(3)$ - the complexified $6$-sphere).\\ 
\indent 
Indeed, this follows from the fact that $Q'\setminus Q$ is the complexified $6$-sphere.\\ 
\indent 
Now, let $p\in Q'$ and denote by $\a$ the linear map from $\Lambda^2p$ to $U$ given by the octonionic cross product. 
It follows that if $p\in Q$ then $\a$ has rank $1$ and its image is contained by $p$\,, whilst, if $p\in Q'\setminus Q$ then the rank of $\a$ is $3$ 
and if we denote by $q$ its image then $p$ and $q$ are the $(\pm{\rm i})$-eigenspaces of a unit vector orthogonal to $p+q$\,. 
\end{proof} 

\begin{rem} \label{rem:3_isotropic} 
The embedding of $Q$ into $Q'$ is determined by the octonionic cross product as follows. Let $\ell\in Q(\subseteq PU)$, 
and choose a nondegenerate associative space $p$ that contains $\ell$\,. Then the image of $\ell$ into $Q'$ is equal to $\ell+(\ell\times p^{\perp})$\,, 
where $\times$ denotes the octonionic cross product.  
\end{rem} 

\indent 
We call the points of $Q\subseteq Q'$ \emph{isotropic associative spaces of dimension $3$}\,. 

\begin{cor} \label{cor:2_isotropic} 
{\rm (i)} Any self-dual space (contained by a nondegenerate coassociative space) is contained by a unique isotropic associative space of dimension $3$\,.\\  
\indent 
{\rm (ii)} Any anti-self-dual space is contained by a family of isotropic associative spaces of dimension $3$\,, parametrized by the projective line.\\ 
\indent 
{\rm (iii)} Any isotropic associative space of dimension $3$ containes a family of anti-self-dual spaces, parametrized by the projective line.   
\end{cor} 
\begin{proof}
(i) If $p$ is self-dual then $p\times p$ has dimension $1$ and determines (as in Remark \ref{rem:3_isotropic}\,) 
an isotropic associative space, of dimension $3$\,, containing $p$\,.\\   
\indent 
(ii) Any anti-self-dual space is of the form $\ell+x\times\ell$ where $\ell$ is an isotropic direction 
and $x\in U$ is nondegenerate and orthogonal onto a nondegenerate associative space containing $\ell$\,. 
Consequently, if $p$ is anti-self-dual then it is contained by the isotropic associative spaces of dimension $3$\,, determined by the directions contained by $p$\,.\\ 
\indent 
(iii) By (i) any isotropic associative space of dimension $3$ is of the form $\ell+p$ where $\ell$ is an isotropic direction and $p$ is self-dual 
such that $\ell=p\times p$\,. Then any anti-self-dual space contained by $\ell+p$ is generated by $\ell$ and a direction in $p$\,.   
\end{proof} 

\begin{thm} \label{thm:isotropic_associative} 
Let $p\subseteq U$ be an isotropic subspace of dimension $2$\,. Then the following assertions hold:\\ 
\indent 
{\rm (i)} $p$ is self-dual if and only if $p\times p\neq\{0\}$\,.\\ 
\indent 
{\rm (ii)} $p$ is anti-self-dual if and only if $p\times p=\{0\}$\,. 
\end{thm} 
\begin{proof}
(i) If $p\times p\subseteq p$ then we can choose a basis $(a,b)$ of $p$ such that $a\times b=a$\,. Furthermore, any isotropic subpace $q$\,, of dimension $3$\,, 
containing $p$ is associative (as $q\cap(q\times q)\neq\{0\}$\,). Thus, we can choose a basis $(x,y,z)$ of $q$ such that $x\times y=z$\,, $y\times z=0$ and $z\times x=0$ 
(consequence of Remark \ref{rem:3_isotropic}\,).\\ 
\indent 
Now, let $(\a,\b,\g)$ and $(\a',\b',\g')$ be the components of $a$ and $b$\,, respectively, with respect to $(x,y,z)$\,. Then 
$a\times b=(\a\b'-\a'\b)z$ which, together with $a=a\times b$\,, implies $\a=\b=0$\,, and, hence, $a\times b=0$\,, a contradiction. 
Therefore $p+(p\times p)$ is isotropic associative of dimension $3$\,.\\ 
\indent 
(ii) If $q\subseteq U$ is isotropic, of dimension $3$\,, containing $p$ then $q\times q$ has dimension at most $2$\,. Thus, $q$ is associative, 
determined by the isotropic direction $q\times q$\,. Hence, $q\times q\subseteq p$ (as, otherwise, we would have $p\times p\neq\{0\}$\,). 
The proof follows. 
\end{proof}

\begin{rem} 
1) Let $p,q\subseteq U$ be isotropic associative of dimension $3$\,. Then either $p=q$\,, or $p\cap q=\{0\}$\,, or $p\cap q$ is an anti-self-dual space.\\ 
\indent 
2) Let $V\subseteq U$ be a nondegenerate vector subspace of codimension $1$\,, and let $p\subseteq U$ be the isotropic associative space of dimension $3$ 
determined by the isotropic direction $\ell\subseteq U$. Then one and only one of the following occurs:\\ 
\indent 
\quad(i) $p\subseteq V$;\\ 
\indent 
\quad(ii) $p\cap V$ is anti-self-dual;\\ 
\indent 
\quad(iii) $p\cap V$ is self-dual.\\ 
\indent 
To understand when each case occurs, endow $U$ with a conjugation preserving $V$ and compatible with the octonionic cross product. 
Then there exists a unique real associative space $q$ that contains $\ell$ and the following hold:\\ 
\indent 
\quad(a) Case (i) occurs if and only if $q\supseteq V^{\perp}$ and $\ell\subseteq V$;\\ 
\indent 
\quad(b) Case (ii) occurs if and only if $q\subseteq V$ (and $\ell\subseteq V$);\\ 
\indent 
\quad(c) Case (iii) occurs if and only if $q\nsubseteq V$ and $\ell\nsubseteq V$.\\ 
\indent 
Consequently, we, also, have the following:\\ 
\indent 
\quad$\bullet$ $p\subseteq V$ if and only if $\ell\subseteq V$ and contained by an eigenspace of the orthogonal complex structure on $V$ (given by the 
octonionic cross product);\\ 
\indent 
\quad$\bullet$ $p\cap V$ is anti-self-dual if and only if $\ell\subseteq V$ but not contained by an eigenspace of the orthogonal complex structure on $V$;\\ 
\indent 
\quad$\bullet$ $p\cap V$ is self-dual if and only if $\ell\nsubseteq V$. 
\end{rem}

\section{Obstructions to the integrability of almost twistorial structures} \label{section:aG2_atwist} 

\indent 
We continue with the same notations as in Section \ref{section:iso_octo}\,. Let $E$ be the tautological vector bundle over $Q$ given by its embedding  
into $Q'$. Then the tautological line bundle $L$ over $Q\,(\subseteq PU)$ is a subbundle of $E$ and, on denoting $U_+=E/L$\,, we have an equivariant 
exact sequence of homogeneous vector bundles 
\begin{equation} \label{e:first_exact_seq} 
0\longrightarrow L\longrightarrow E\longrightarrow U_+\longrightarrow0\;. 
\end{equation} 
\indent 
By using \cite{Bott-BBW}\,, we duce that $H^j(U_+^*)=0$\,, for any $j\in\mathbb{N}$\,. Together with the Kodaira vanishing theorem and by 
passing to the exact sequence of cohomology groups of the dual of \eqref{e:first_exact_seq}\,, this implies $H^0(E^*)=U_{1,0}$ (equivariantly), 
and $H^j(E^*)=0$\,, for any $j\in\mathbb{N}\setminus\{0\}$\,, where $U_{m,n}$ is the irreducible representation space of $G_2$ corresponding to $(m,n)\in\mathbb{N}^2$ 
(note that, $U_{1,0}=U$ the space of imaginary octonions).\\ 
\indent 
We, also, have the following equivariant exact sequence 
\begin{equation} \label{e:fundamental_exact_seq} 
0\longrightarrow E^{\perp}\longrightarrow Q\times U_{1,0}\longrightarrow E^*\longrightarrow0\;. 
\end{equation} 
Hence, $H^j(E^{\perp})=0$\,, for any $j\in\mathbb{N}$\,.
Furthermore, by restricting $E^{\perp}/E$ to any associative conic we obtain that $E^{\perp}/E$ is a trivial line bundle. 
Thus, we have an equivariant exact sequence $0\longrightarrow Q\times U_{0,0}\longrightarrow (E^{\perp})^*\longrightarrow E^*\longrightarrow 0$\,, 
from which we deduce $H^0((E^{\perp})^*)=U_{0,0}\oplus U_{1,0}$\,, and $H^j((E^{\perp})^*)=0$\,, for any $j\in\mathbb{N}\setminus\{0\}$\,. Consequently, 
dualizing \eqref{e:fundamental_exact_seq}\,, passing to the cohomology exact sequence, and then dualizing again, we obtain (cf.~\cite{Pan-qgfs}\,) 
\begin{equation} \label{e:fundamental_exact_seq_2} 
0\longrightarrow E^{\perp}\longrightarrow Q\times(U_{0,0}\oplus U_{1,0})\longrightarrow(E^{\perp})^*\longrightarrow0\;,  
\end{equation} 
together with the obvious morphism from \eqref{e:fundamental_exact_seq_2} to \eqref{e:fundamental_exact_seq}\,.
These two equivariant exact sequences give the Euclidean twistorial structures we are interested in.\\ 
\indent 
From Section \ref{section:iso_octo} it follows that $G_2$ has two orbits on ${\rm Gr}_2^0(U)$\,: the space $Y$ of anti-self-dual spaces and the space of self-dual spaces.  
The former is just the twistor space of the Wolf space determined by $G_2$ (the closed adjoint orbit into $P\mathfrak{g}_2$); in particular, $Y$ is compact.\\ 
\indent 
Let $U_-$ be the tautological vector bundle over $Y$. Denote $Z=PU_+=PU_-$ and let $\p_{\pm}$ be the projections from $Z$ onto $Q$ and $Y$, respectively. 
Let $L_+=\p_+^*L$ and, similarly, let $L_-$ be the pull-back by $\p_-$  of the tautological line bundle over $Y$ (given by the Pl\"ucker embedding). 

\begin{prop} 
{\rm (i)} $L_+$ is the tautological line bundle over $PU_-$\,.\\ 
\indent 
{\rm (ii)} $L_-$ is the tautological line bundle over $P(L_+\otimes U_+)$\,.
\end{prop}   
\begin{proof}
Assertion (i) is obvious. To prove (ii)\,, note that, $\p_+^*E/\p_-^*(U_-)$ is the dual of the tautological line bundle over $P(U_+^*)$\,. 
As $U_+^*=(\Lambda^2U_+^*)\otimes U_+=L_+^*\otimes U_+$ and $\p_+^*E/\p_-^*(U_-)=L_+^2\otimes L_-^*$\,, the proof follows quickly. 
\end{proof} 

\begin{cor} 
For any $m,n\in\mathbb{N}$\,, we have $U_{m,n}=H^0\bigl((L_+^*)^m\otimes(L_-^*)^n\bigr)$. 
\end{cor} 
\begin{proof}
By using a description of $\mathfrak{g}_2$ from \cite{Sal-holo_book} (see \cite{Pan-eholon}\,) we deduce that $Z$ is the generalized complete flag manifold of $G_2$\,. 
The proof follows from the Borel-Weil theorem.  
\end{proof}

\begin{prop} \label{prop:H^0Lambda^2} 
{\rm (i)} $H^0(\Lambda^2E^*)=U_{1,0}\oplus U_{0,1}$\,.\\ 
\indent 
{\rm (ii)} $H^0(\Lambda^2(E^{\perp})^*)=U_{1,0}\oplus U_{1,0}\oplus U_{0,1}$\,.
\end{prop} 
\begin{proof}
Obviously, $\p_-^*(U_-)$ is a subbundle of $\p_+^*E$. Also, the projection from $\Lambda^2U_{1,0}$ onto $U_{0,1}$ decomposes as 
the composition of an equivariant linear map from the former to $H^0(\Lambda^2E^*)$ followed by the (equivariant) linear map between the 
spaces of sections of $\Lambda^2E^*$ and $\Lambda^2U_-^*$.  
Furthermore, the kernel of the linear map between the spaces of sections of $\Lambda^2E^*$ and $\Lambda^2U_-^*$ is formed of the sections of $\Lambda^2U_+^*=L^*$, 
and the proof of (i) quickly follows.\\  
\indent 
Assertion (ii) follows from (i) and the fact that $E^*$ is the kernel of the vector bundles morphisms from $\Lambda^2(E^{\perp})^*$ onto $\Lambda^2E^*$. 
\end{proof} 

\begin{rem} \label{rem:Lambda^2} 
We can improve Proposition \ref{prop:H^0Lambda^2} as follows. Firstly, as $\rank E=3$\,, we have $\Lambda^2E^*=(\Lambda^3E^*)\otimes E=(L_+^*)^2\otimes E$\,. 
Together with \eqref{e:first_exact_seq}\,, this gives $$0\longrightarrow L_+^*\longrightarrow(L_+^*)^2\otimes E\longrightarrow(L_+^*)^2\otimes U_+\longrightarrow0\;,$$  
whose cohomology exact sequence gives (i) of Proposition \ref{prop:H^0Lambda^2}\,, by using \cite{Bott-BBW}\,; moreover, this way, we, also, obtain that  
$H^j(\Lambda^2E^*)=0$\,, for any $j\geq1$\,.\\ 
\indent 
Consequently, from 
\begin{equation} \label{e:for_Lambda^2E^perp}
0\longrightarrow E^*\longrightarrow\Lambda^2(E^{\perp})^*\longrightarrow(L_+^*)^2\otimes E\longrightarrow0
\end{equation} 
we, similarly, obtain (ii) of Proposition \ref{prop:H^0Lambda^2} and $H^j(\Lambda^2(E^{\perp})^*)=0$\,, for any $j\geq1$\,.
\end{rem} 

\begin{thm} \label{thm:for_torsion} 
We have 
\begin{equation} \label{e:for_torsion} 
\begin{split} 
H^0\bigl(E^*\otimes\Lambda^2(E^{\perp})^*\bigr)&=U_{1,0}\oplus kU_{0,1}\oplus3U_{2,0}\oplus U_{1,1}\;,\\ 
H^0\bigl((E^{\perp})^*\otimes\Lambda^2(E^{\perp})^*\bigr)&=3U_{1,0}\oplus(k+1)U_{0,1}\oplus3U_{2,0}\oplus U_{1,1}\;,\\
\end{split} 
\end{equation}
where $k\in\{2,3\}$\,. 
\end{thm} 
\begin{proof}
We have $\otimes^2E^*=(\odot^2E^*)\oplus(\Lambda^2E^*)$\,, where $\odot$ denotes the symmetric product.\\ 
\indent  
From \eqref{e:first_exact_seq} we obtain that $L_+\otimes E$ is the kernel of $\odot^2E\to\odot^2U_+$\,, and, consequently, we have the following exact sequence 
\begin{equation} \label{e:for_odot_E^*} 
0\longrightarrow\odot^2U_+^*\longrightarrow\odot^2E^*\longrightarrow L_+^*\otimes E^*\longrightarrow0\;. 
\end{equation}
\indent 
By using \cite{Bott-BBW}\,, we obtain that $H^1(\odot^2U_+^*)=U_{0,1}$ and $H^j(\odot^2U_+^*)=0$\,, for any $j\neq1$\,.\\ 
\indent 
Similarly, from $0\longrightarrow L_+^*\otimes U_+^*\longrightarrow L_+^*\otimes E^*\longrightarrow(L_+^*)^2\longrightarrow0$ 
(consequence of \eqref{e:first_exact_seq}\,), we deduce that $H^0(L_+^*\otimes E^*)=U_{0,1}\oplus U_{2,0}$ and $H^j(L_+^*\otimes E^*)=0$\,, for any $j\geq1$\,.\\ 
\indent 
Therefore the cohomology exact sequence of \eqref{e:for_odot_E^*} gives 
$$0\longrightarrow H^0(\odot^2E^*)\longrightarrow U_{0,1}\oplus U_{2,0}\longrightarrow U_{0,1}\longrightarrow H^1(\odot^2E^*)\longrightarrow0\;,$$ 
and $H^j(\odot^2E^*)=0$\,, for any $j\geq2$\,. Thus, either $H^1(\odot^2E^*)=0$ and $H^0(\odot^2E^*)=U_{2,0}$\,, 
or $H^1(\odot^2E^*)=U_{0,1}$ and $H^0(\odot^2E^*)=U_{0,1}\oplus U_{2,0}$\,.\\ 
\indent 
Together with Remark \ref{rem:Lambda^2}\,, this gives $H^j(\otimes^2E^*)=0$\,, for any $j\geq2$\,, and only one of the following\\  
\indent 
\quad(1) $H^1(\otimes^2E^*)=0$ and $H^0(\otimes^2E^*)=U_{1,0}\oplus U_{0,1}\oplus U_{2,0}$\,,\\ 
\indent 
\quad(2) $H^1(\otimes^2E^*)=U_{0,1}$ and $H^0(\otimes^2E^*)=U_{1,0}\oplus2U_{0,1}\oplus U_{2,0}$\,.\\ 
\indent 
Further, by tensorising the dual of \eqref{e:first_exact_seq} with $(L_+^*)^2\otimes U_+$ we obtain 
\begin{equation} \label{e:for_(L_+^*)^2_U_+_E^*} 
0\longrightarrow(L_+^*)^2\otimes U_+\otimes U_+^*\longrightarrow(L_+^*)^2\otimes U_+\otimes E^*\longrightarrow(L_+^*)^3\otimes U_+\longrightarrow0\;.
\end{equation}
\indent 
As $(L_+^*)^2\otimes U_+\otimes U_+^*=L_+^*\otimes U_+^*\otimes U_+^*=L_+^*\otimes\bigl((\odot^2U_+^*)\oplus L_+^*\bigr)=\bigl(L_+^*\otimes(\odot^2U_+^*)\bigr)\oplus(L_+^*)^2$\,, 
by using \cite{Bott-BBW}\,, we deduce $H^0\bigl((L_+^*)^2\otimes U_+\otimes U_+^*\bigr)=U_{2,0}$\,, $H^1\bigl((L_+^*)^2\otimes U_+\otimes U_+^*\bigr)=U_{0,0}$\,, 
and $H^j\bigl((L_+^*)^2\otimes U_+\otimes U_+^*\bigr)=0$\,, for any $j\geq2$\,.\\ 
\indent 
Similarly, $H^0\bigl((L_+^*)^3\otimes U_+\bigr)=U_{1,1}$ and $H^j\bigl((L_+^*)^3\otimes U_+\bigr)=0$\,, for any $j\geq1$\,.\\ 
\indent 
Therefore from \eqref{e:for_(L_+^*)^2_U_+_E^*} we obtain 
$$0\longrightarrow U_{2,0}\longrightarrow H^0\bigl((L_+^*)^2\otimes U_+\otimes E^*\bigr)\longrightarrow U_{1,1}\longrightarrow U_{0,0}\longrightarrow
H^1\bigl((L_+^*)^2\otimes U_+\otimes E^*\bigr)\longrightarrow0\;,$$ 
and $H^j\bigl((L_+^*)^2\otimes U_+\otimes E^*\bigr)=0$\,, for any $j\geq2$\,; consequently, 
\begin{equation} \label{e:H^j((L_+^*)^2_U_+_E_*)} 
\begin{split} 
H^1\bigl((L_+^*)^2\otimes U_+\otimes E^*\bigr)&=U_{0,0}\;,\\ 
H^0\bigl((L_+^*)^2\otimes U_+\otimes E^*\bigr)&=U_{2,0}\oplus U_{1,1}\;. 
\end{split} 
\end{equation}
\indent 
Also, by using \eqref{e:first_exact_seq} and \cite{Bott-BBW}\,, we obtain  
\begin{equation} \label{e:H^j(L_+^*_E^*)} 
\begin{split} 
H^0(L_+^*\otimes E^*)&=U_{0,1}\oplus U_{2,0}\;,\\ 
H^j(L_+^*\otimes E^*)&=0\;,\;\textrm{for any}\;j\geq1\;. 
\end{split} 
\end{equation}   
\indent 
By tensorising \eqref{e:first_exact_seq} with $(L_+^*)^2\otimes E^*$ we obtain the exact sequence 
$$0\longrightarrow L_+^*\otimes E^*\longrightarrow(L_+^*)^2\otimes E\otimes E^*\longrightarrow(L_+^*)^2\otimes U_+\otimes E^*\longrightarrow0\;,$$ 
which, together with \eqref{e:H^j((L_+^*)^2_U_+_E_*)} and \eqref{e:H^j(L_+^*_E^*)}\,, gives the following two exact sequences 
$$0\longrightarrow U_{0,1}\oplus U_{2,0}\longrightarrow H^0\bigl((L_+^*)^2\otimes E\otimes E^*\bigr)\longrightarrow U_{2,0}\oplus U_{1,1}\longrightarrow0\;,$$ 
$$0\longrightarrow H^1\bigl((L_+^*)^2\otimes E\otimes E^*\bigr)\longrightarrow U_{0,0}\longrightarrow0\;.$$ 
\indent 
We have, thus, proved the following 
\begin{equation} \label{e:H^j((L_+^*)^2_E_E^*)} 
\begin{split} 
H^0\bigl((L_+^*)^2\otimes E\otimes E^*\bigr)&=U_{0,1}\oplus2U_{2,0}\oplus U_{1,1}\;,\\ 
H^1\bigl((L_+^*)^2\otimes E\otimes E^*\bigr)&=U_{0,0}\;,\\ 
H^j\bigl((L_+^*)^2\otimes E\otimes E^*\bigr)&=0\;,\;\textrm{for any}\;j\geq2\;. 
\end{split} 
\end{equation} 
\indent 
Now, by tensorising \eqref{e:for_Lambda^2E^perp} with $E^*$ we obtain 
\begin{equation} \label{e:for_T}
0\longrightarrow\otimes^2E^*\longrightarrow E^*\otimes\Lambda^2(E^{\perp})^*\longrightarrow(L_+^*)^2\otimes E\otimes E^*\longrightarrow0\;. 
\end{equation} 
\indent 
From the cohomology exact sequence of \eqref{e:for_T}\,, together with \eqref{e:H^j((L_+^*)^2_E_E^*)} and the two possibilities (1) and (2)\,, above, 
it follows the first relation of \eqref{e:for_torsion}\,.\\ 
\indent 
To prove the second relation of \eqref{e:for_torsion}\,, we use the exact sequence 
$0\longrightarrow Z\times U_{0,0}\longrightarrow(E^{\perp})^*\longrightarrow E^*\longrightarrow0$\,, 
which implies 
\begin{equation} \label{e:last_for_torsion}  
0\longrightarrow\Lambda^2(E^{\perp})^*\longrightarrow(E^{\perp})^*\otimes\Lambda^2(E^{\perp})^*\longrightarrow E^*\otimes\Lambda^2(E^{\perp})^*\longrightarrow0\;. 
\end{equation} 
\indent 
Finally, the cohomology exact sequence of \eqref{e:last_for_torsion}\,, together with Remark \ref{rem:Lambda^2}\,, and the first relation of \eqref{e:for_torsion}\,,  
quickly completes the proof. 
\end{proof}

\indent 
Let $M$ be a manifold, with $\dim M=7,8$\,, endowed with an almost $G_2$-structure, through the representations $U_{1,0}$ and $U_{0,0}\oplus U_{1,0}$\,, respectively. 
Any compatible connection $\nabla$ on $M$ induces an almost twistorial structure on $M$, by using the Euclidean twistorial structure given by 
\eqref{e:fundamental_exact_seq} and \eqref{e:fundamental_exact_seq_2}\,, respectively, which we call \emph{the canonical almost twistorial structure of $(M,\nabla)$}\,. 

\begin{cor} \label{cor:integrab_G2} 
Let $M$ be a manifold, with $\dim M=7,8$\,, endowed with an almost $G_2$-structure, through the representations $U_{1,0}$ and $U_{0,0}\oplus U_{1,0}$\,, respectively. 
Let $\nabla$ be a compatible connection on $M$, and denote by $R$ and $T$ its curvature form and torsion tensor field, respectively.\\ 
\indent 
Then the following assertions are equivalent:\\ 
\indent 
\quad{\rm (i)} The canonical almost twistorial structure of $(M,\nabla)$ is integrable and the Penrose-Ward transformation can be applied to it (locally).\\ 
\indent 
\quad{\rm (ii)} $R=0$ and $T\in kU_{0,0}\oplus kU_{1,0}$\,, at each point, where $k=1,2$ if $\dim M=7,8$\,, respectively. 
\end{cor} 
\begin{proof}
We prove the $\dim M=7$ case, the proof of the other one is similar.\\ 
\indent 
By using \cite[2.7.3]{Bry-G2} we obtain 
\begin{equation} \label{e:for_torsion_integrab_G2}
(\Lambda^2U_{1,0})\otimes U_{1,0}=U_{0,0}\oplus2U_{1,0}\oplus U_{0,1}\oplus2U_{2,0}\oplus U_{1,1}\;. 
\end{equation}
\indent 
Assertion (i) is equivalent to the fact that, at each point, the sections given by restricting $R$ and $T$ to 
$\Lambda^2E^{\perp}$ and $\Lambda^2(E^{\perp})\otimes E$\,, respectively, are zero (here $R$ is seen as a bundle valued $2$-form, and $T$ as the tensor field of degree $(0,3)$ 
given by the torsion of $\nabla$ and the underlying Riemannian metric on $M$).\\ 
\indent 
The proof follows from Proposition \ref{prop:H^0Lambda^2}\,, Theorem \ref{thm:for_torsion}\,, and \eqref{e:for_torsion_integrab_G2}\,. 
\end{proof}

\begin{rem} \label{rem:corrigenda} 
Corollary \ref{cor:integrab_G2} corrects statements of \cite[\S2]{Pan-eholon}\,, and the results therein can be straightforwardly formulated 
in the current setting.  
\end{rem}

\end{document}